\documentclass[10pt,a4paper]{amsart}
\usepackage{amsmath,amsthm,amssymb,mathrsfs,hyperref}
\usepackage{aliascnt}
\usepackage{tgtermes}
\usepackage[utf8]{inputenc}
\usepackage[T1]{fontenc}
\usepackage{fullpage}

\usepackage{xcolor}
\definecolor{doom}{rgb}{0,0,0.68}
\hypersetup{
	unicode=true,
	colorlinks=true,
	citecolor=doom,
	linkcolor=doom,
	anchorcolor=doom
}

\newtheorem{theorem}{Theorem}
\newaliascnt{example}{theorem}

\aliascntresetthe{example}
\newaliascnt{corollary}{theorem}
\newtheorem{corollary}[corollary]{Corollary}
\aliascntresetthe{corollary}
\newaliascnt{lemma}{theorem}
\newtheorem{lemma}[lemma]{Lemma}
\aliascntresetthe{lemma}

\newtheorem*{remark}{Remark}
\newtheorem*{example*}{Example}
\newtheorem*{conjecture}{Conjecture}
\newtheorem*{question}{Question}

\theoremstyle{definition}
\newaliascnt{definition}{theorem}
\newtheorem{definition}[definition]{Definition}
\aliascntresetthe{definition}

\newcommand{\axiom}[1]{\mathsf{#1}} 
\newcommand{\ZFC}{\axiom{ZFC}}

\newcommand{\NS}{\mathrm{NS}}
\newcommand{\Reg}{\mathrm{Reg}}

\DeclareMathOperator{\cf}{cf}

\DeclareMathOperator{\rng}{rng}
\DeclareMathOperator{\acc}{acc}

\DeclareMathOperator{\pcf}{pcf}

\newcommand{\sdiff}{\mathbin{\triangle}}

\newcommand{\forces}{\mathrel{\Vdash}}
\newcommand{\PP}{\mathbb{P}}
\newcommand{\BB}{\mathbb{B}}
\newcommand{\QQ}{\mathbb{Q}}
\newcommand{\RR}{\mathbb{R}}

\newcommand{\power}{\mathcal{P}}
\newcommand{\tup}[1]{\left\langle#1\right\rangle}
\newcommand{\midd}{\mathrel{}\middle|\mathrel{}}
\newcommand{\res}{{\upharpoonright}}
\newcommand{\E}{\mathrel{E}}
\newcommand{\Los}{{\L}o\'s\ }
\newcommand{\cC}{\mathcal C}

\author{Yair Hayut}
\address[Yair Hayut]{\textbf{Einstein Institute of Mathematics}\newline
Edmond J. Safra Campus,\newline
The Hebrew University of Jerusalem.\newline
Givat Ram, Jerusalem, 91904, Israel}
\email[Yair Hayut]{yair.hayut@mail.huji.ac.il}
\author{Asaf Karagila}
\address[Asaf Karagila]{\textbf{Einstein Institute of Mathematics}\newline
Edmond J. Safra Campus\newline
The Hebrew University of Jerusalem.\newline
Givat Ram, Jerusalem, 91904, Israel}
\email[Asaf Karagila]{karagila@math.huji.ac.il}
\urladdr[Asaf Karagila]{http://boolesrings.org/asafk}
\date{July 28, 2015}

\title{Restrictions on Forcings That Change Cofinalities}

\begin{document}
\begin{abstract}
In this paper we investigate some properties of forcing which can be considered ``nice'' in the context of singularizing regular cardinals to have an uncountable cofinality. We show that such forcing which changes cofinality of a regular cardinal, cannot be too nice and must cause some ``damage'' to the structure of cardinals and stationary sets. As a consequence there is no analogue to the Prikry forcing, in terms of ``nice'' properties, when changing cofinalities to be uncountable.
\end{abstract}
\maketitle

\section{Introduction}\label{Section:Intro}

In this paper we examine a few properties of forcing notions which limit their ability to change cofinalities of regular cardinals without collapsing them. This work is the result of trying to understand whether $\square(\kappa)$ can fail at the first inaccessible cardinal. One naive approach would be to start with a sufficiently rich core model, and singularize many cardinals while preserving the inaccessibility of $\kappa$, and without adding $\square(\kappa)$ sequences. The easiest way to do this would be to ensure the forcing is $\sigma$-closed.

When this approach failed, it soon became apparent that $\sigma$-closed forcings cannot change cofinalities without collapsing cardinals. The natural question that followed is how ``nice'' can be a forcing that changes the cofinality of $\kappa$ without collapsing it? For a countable cofinality we have the Prikry forcing which is weakly homogeneous, does not add bounded subsets to $\kappa$, and does not change the cofinality of any other regular cardinal. Can these properties be duplicated when the target cofinality is uncountable?

In \autoref{Section:Proper} we prove that a proper forcing, in particular a $\sigma$-closed forcing, cannot change cofinalities without collapsing cardinals. In \autoref{Section:PCF} we make some observations related to PCF theory that suggest that changing many cofinalities at once must come at the price of adding $\omega$-sequences. The \hyperref[Section:Reflective]{fourth section} is dedicated to investigating properties related to homogeneity. We show that under some relatively weak additional hypotheses, the existence of such forcing affects the cardinal structure of $V$; it implies that the cardinal that changes its cofinality to uncountable cofinality must be preceded by stationarily many cardinals that change their cofinalities as well. We also show that under some further natural assumptions this cardinal must be at least Mahlo of high order, thus revealing some of its ``largeness'' already in $V$. 

The conclusion we draw from the work presented here is that in order to change the cofinality of a regular cardinal to be uncountable, the forcing cannot have too many nice properties simultaneously. In particular there is no analogue of the Prikry forcing to this case, at least in terms of ``minimal damage'' to the structure of cardinals in the ground model.

Throughout this paper we work in $\ZFC$. Our notation is standard: for a forcing notion $\PP$ and two conditions $p, q \in\PP$, $q \leq p$ means that $q$ is stronger than $p$. All the definitions and properties of stationary sets, proper forcing and PCF, that are needed in this paper, are taken from Jech's ``Set Theory'' (\cite{Jech2003}).

\section{Proper Forcing}\label{Section:Proper}
\begin{theorem}\label{Theorem:proper}
Assume $\kappa$ is regular. Let $\PP$ be a forcing such that $\forces_\PP\cf(\check\kappa)=\check\mu>\check\omega$. Moreover assume there is a partition of $\kappa$ into stationary sets, $\{S_\alpha\mid\alpha<\kappa\}$ , all of which remain stationary after forcing which $\PP$. Then $\forces_\PP |\check\kappa|=|\check\mu|$.
\end{theorem}
\begin{proof}
Fix in $V$ a partition of $\kappa$ as in the assumptions of the theorem. Let $G\subseteq\PP$ be a $V$-generic filter, and $s\in V[G]$ a continuous and cofinal function from $\mu$ to $\kappa$. We define $f\colon\mu\to\kappa$ by $f(\beta)=\alpha$ if and only if $s(\beta)\in S_\alpha$. Since $\rng s$ is a club in $\kappa$ and the stationarity of each $S_\alpha$ is preserved in $V[G]$, $\rng s\cap S_\alpha\neq\varnothing$ for all $\alpha<\kappa$, and therefore $f$ is onto $\kappa$.
\end{proof}

We remark that if $\PP$ is a forcing preserving the stationarity of all subsets of a stationary $S\subseteq\kappa$, then such partition exists trivially by partitioning $S$ and adding $\kappa\setminus S$ to one of the parts. Proper forcings preserve the stationarity of subset of $S^\kappa_\omega$, which is the key point in the proof of the following corollary.

\begin{corollary}\label{Corollary:Proper}
If $\PP$ is a proper forcing which changes the cofinality of a regular cardinal $\kappa$, then $\PP$ collapses $\kappa$ to its new cofinality.
\end{corollary}
\begin{proof}
If $\PP$ changes the cofinality of $\kappa$ to $\omega$, let $G$ be a generic filter. In $V[G]$ take $s$ a cofinal sequence in $\kappa$, of order type $\omega$. Then $[\kappa]^\omega$, $\{A\in[\kappa]^\omega\mid s\subseteq A\}$ is a club in $[\kappa]^\omega$. But this club is disjoint of every subset of $[\kappa]^\omega\cap V$. Therefore $\PP$ is not proper.

If $\PP$ is proper then it preserves the stationarity of subsets of $S^\kappa_\omega$. If $\PP$ changes the cofinality of $\kappa$ to be uncountable, then $\kappa$ is collapsed, as shown in the theorem above.
\end{proof}

\begin{corollary}
If $\PP$ is $\sigma$-closed and $\PP$ changes the cofinality of a regular cardinal to $\omega_1$, then it collapses it.\qed
\end{corollary}
One might look at the definition of the Prikry forcing, and attempt to naively generalize it to uncountable cofinalities by using  closed initial segments of a generic club. For an uncountable target cofinality such forcing would be $\sigma$-closed and the above corollary means that it cannot possibly work. This means that the definition of Magidor and Radin forcing using finite approximation is in fact an inescapable reality. We finish this section with a corollary showing the usefulness of this theorem in a context where cardinals are not necessarily preserved.

\begin{corollary}\label{Corollary:Singulars}
If $\PP$ is a $\sigma$-closed forcing, $\lambda$ singular and $\PP$ collapses $\lambda^+$, then $\PP$ collapses $\lambda$. In particular it is impossible to change the successor of a singular cardinal with a $\sigma$-closed forcing.\qed
\end{corollary}
\begin{proof}
Since $\lambda^+$ is not a cardinal, its cofinality is some $\mu<\lambda$. By $\sigma$-closure it follows that $\mu>\omega$, and so $\lambda^+$ is collapsed to $\mu$, and so is $\lambda$.
\end{proof}

\section{PCF related restrictions}\label{Section:PCF}

Gitik proved in \cite[Section~3]{Gitik1986} that it is consistent that there is a model in which $\kappa$ is regular, and there is a forcing which changes the cofinality of $\kappa$ to be $\omega_1$ without adding bounded subsets. In particular the forcing does not change cofinalities below $\kappa$, or the cofinalities of other regular cardinals. Gitik in fact showed it is possible for any regular cardinal $\mu<\kappa$ to be the new cofinality of $\kappa$, starting from a $\kappa$ which is measurable of Mitchell order $\mu$.

Assume there are measurable cardinals $\kappa_1,\ldots,\kappa_n$ of Mitchell order $\omega_1$. It is possible to extend Gitik's argument by induction to obtain a model where each $\kappa_i$ is regular, and there is a forcing that changes their cofinality to be $\omega_1$, without changing the cofinalities of other regular cardinals.

As Gitik's forcing is $\kappa^+$-c.c., adds no bounded subsets of $\kappa$, and the cofinality of $\kappa$ is changed to be $\omega_1$, it can be shown that this forcing is $\sigma$-distributive. It follows, if so, that $\sigma$-distributivity can be preserved when extending the argument by induction. Naturally we may ask if this can be extended to infinitely many simultaneous cofinality changes. The next theorem shows that if a positive answer is at all consistent, it will require significantly stronger large cardinal assumptions.
\begin{theorem}\label{theorem:pcf}
Let $\tup{\kappa_n\mid n<\omega}$ be an increasing sequence of regular cardinals, and let $\PP$ be a forcing such that for every $n<\omega$, $\forces_\PP\cf(\check\kappa_n)=\check\mu$. Then either $\PP$ adds an $\omega$-sequence of ordinals, or for every $\lambda\in\pcf(\{\kappa_n\mid n<\omega\})$, $\forces_\PP\cf(\check\lambda)=\check\mu$.
\end{theorem}  
\begin{proof}
Let $G\subseteq\PP$ be a $V$-generic filter. Suppose that $\PP$ does not add any $\omega$-sequences; in particular this implies that $\mu>\omega$. We will show that in $V[G]$ there is a scale of length $\mu$ in $\prod_{n<\omega}\kappa_n$; from this will follow that if $U$ is an ultrafilter on $\omega$, then the cofinality of $(\prod_{n<\omega}\kappa_n)/U$ is $\mu$. Therefore if $\lambda$ was the cofinality of the product in $V$, it must be that $V[G]\models\cf(\lambda)=\mu$.

We did not add any $\omega$-sequences, and therefore $\prod_{n<\omega}\kappa_n$ is the same object in $V$ and $V[G]$. Additionally, if $U$ is an ultrafilter on $\omega$ in $V$, then $U$ is still an ultrafilter in $V[G]$.

Let $\gamma_n\colon\mu\to\kappa_n$ be a continuous and cofinal function, and denote by $g_\alpha(n)=\gamma_n(\alpha)$. We have that for each $\alpha<\mu$, $g_\alpha\in V$ as it is an $\omega$-sequence of ordinals. Consider $f\in\prod_{n<\omega}\kappa_n$, for each $n$ there exists $\alpha<\mu$ such that $f(n)<g_\alpha(n)$. Therefore there is some $\alpha$ such that for all $n$, $f(n)<g_\alpha(n)$, as wanted.
\end{proof}

\begin{remark}The assumption that $\PP$ does not add an $\omega$-sequence can be replaced by the assumption that $\mu$ is uncountable and $(\prod_{n<\omega}\kappa_n)^V$ is bounding in $(\prod_{n<\omega}\kappa_n)^{V[G]}$.
\end{remark}

It seems evident from Shelah's covering theorem \cite[Chapter~7,~Lemma~4.9]{ShelahCAbook} that changing all the cofinalities in $\pcf(\{\kappa_n\mid n\in\omega\})$ to be $\mu$, without collapsing cardinals and without adding $\omega$-sequences of ordinals   will require very large cardinals, much larger than $o(\kappa)=\omega_1$.

\begin{corollary}
Let $\tup{\kappa_n\mid n < \omega}$ be a sequence of regular cardinals, $\kappa_\omega=\sup\{\kappa_n\mid n<\omega\}$, and $\PP$ a forcing changing the cofinality of each $\kappa_n$ to be $\omega_1$ and preserves cofinalities of regular cardinals above $\kappa_\omega$. Then $\PP\times\PP$ collapses $\omega_1$.
\end{corollary}
\begin{proof}
Using the notation of the previous proof, in $V[G]$ we know that the $g_\alpha$'s added by $\PP$ have the property that for some large enough $\alpha$, $g_\alpha$ is not bounded by any function in $(\prod_n\kappa_n/U)^V$, where $U\in V$ is some ultrafilter such that $\cf(\prod_n\kappa_n/U)=\lambda>\kappa_\omega$ (note that $U$ might not be an ultrafilter in $V[G]$ and the reduced product, as calculated in $V[G]$ might not be a linear order anymore, which is exactly why we cannot say that $g_\alpha$ is bounding). If that did not happen, then $\lambda$ had changed its cofinality, contrary to the assumption.

Let $\dot g_\alpha$ be a $\PP$-name for $g_\alpha$, and let $H_0\times H_1$ be a $V$-generic filter for $\PP\times\PP$. If $\omega_1$ is not collapsed in $V[H_0\times H_1]$, then on a club $C\subseteq\omega_1$ we have that $\dot g_\alpha^{H_0}$ and $\dot g_{\alpha}^{H_1}$ are identical, for $\alpha\in C$. 

For every $\alpha$ large enough, $\dot g_\alpha^{H_0}$ is not bounded by the product $(\prod_{n<\omega} \kappa_n / U)^V$ and in particular it is not in $V$. So there is a countable ordinal $\beta < \omega_1$ and a condition $p_0\in H_0$ such that $p_0 \forces_\PP \forall \alpha > \check{\beta}:\dot g_\alpha\notin V$. The same holds for $H_1$ (maybe with a larger ordinal $\beta$), so let us pick a condition $\tup{p_0, p_1}\in H_0\times H_1$ that forces that for every $\alpha > \beta$, $\dot{g}_\alpha^{H_0}, \dot{g}_\alpha^{H_1}\notin V$. 

Since $C$ is forced to be unbounded, there is a condition $\tup{q_0, q_1} \leq_{\PP\times\PP} \tup{p_0, p_1}$ in $H_0\times H_1$ that forces $\gamma \in C$ for some $\gamma > \beta$. In particular it forces that $\dot{g_\gamma}^{H_0} = \dot{g_\gamma}^{H_1} \notin V$. But this is impossible, since $H_0, H_1$ are mutually generic and therefore realize differently every name of a new set.   
\end{proof}
Note that virtually the same proof shows that this result holds also for different forcing notions $\PP$ and $\QQ$. If both $\PP$ and $\QQ$ change the cofinality of each $\kappa_n$ to $\omega_1$ while preserving cofinalities above $\kappa_\omega$, then $\PP\times \QQ$ collapses $\omega_1$. 
\begin{corollary}
Let $\PP$ be Magidor forcing for adding a cofinal sequence of order type $\omega_1\cdot\omega$ to a measurable cardinal of Mitchell order $\omega_1 + 1$. Then $\PP\times\PP$ collapses $\omega_1$. 
\end{corollary}
\begin{proof} We give a sketch of the proof.
Let $\kappa$ be the measurable cardinal with $o(\kappa) = \omega_1 + 1$ for which we are going to add a cofinal sequence of order type $\omega_1 \cdot \omega$. Note that although $\PP$ changes the cofinality of $\omega$ many cardinals to $\omega_1$ (the cardinals in the places $\omega_1, \omega_1 \cdot 2,\dots$ in the Magidor sequence), since this sequence is not in the ground model we cannot apply the previous corollary out of the box. 

We overcome this difficulty by observing that $\PP$ can be decomposed into an iteration of a Prikry forcing and a forcing that changes the cofinality of the cardinals in the Prikry sequence to $\omega_1$. The first part adds a Prikry sequence relative to a measure $U\subseteq\mathcal P(\kappa)$ such that $o(U)=\omega_1$, without adding any bounded sets. Let $\PP = \QQ \ast\dot \RR$ where $\QQ$ is the Prikry forcing. First we observe that $\PP$ is $\kappa$-centered, therefore $\PP\times\PP$ is $\kappa$-centered as well, and so $\PP\times\PP$ does not change any cofinalities above $\kappa$. Consequently neither does any of these forcings.

Let $H_1 \times H_2$ be a generic filter for $\QQ \times \QQ$. Then in $V[H_1\times H_2]$ there are forcing notions $\dot\RR^{H_1}$, $\dot\RR^{H_2}$ and each of them changes the cofinality of a cofinal sequence of cardinals to $\omega_1$. By density arguments, there are unboundedly many cardinals that appear in both the Prikry sequence of $H_1$ and $H_2$, so we can pick a cardinal in both sequences, and apply the previous corollary.
\end{proof}
\begin{remark} For every $\alpha \leq \omega_1$, if $\PP$ is the Magidor forcing for a measurable of Mitchell order $\alpha$, then $\PP\times \PP$ preserves all cardinals.
\end{remark}
\section{Amalgamable and Reflective Forcing}\label{Section:Reflective}

Two properties we consider ``nice'', or well-behaved, of a forcing which adds subsets to a cardinal $\kappa$ are: weak homogeneity, and not adding bounded subsets to $\kappa$. The Prikry forcing is indeed very nice in this aspect, while changing the cofinality of $\kappa$ to $\omega$. We want to know about the interaction between homogeneity, not adding bounded sets to $\kappa$, and changing the cofinality of $\kappa$ to be uncountable. We do this in this section by weakening the first two properties and investigating them instead.

Recall that a forcing $\BB$ is called \textit{weakly homogeneous} if for every $p,q\in\BB$ there is an automorphism $\pi$ of $\BB$ such that $\pi q$ is compatible with $p$. We are only interested in the case where $\BB$ is a complete Boolean algebra, since there are $\PP$ and $\PP'$ such that $\PP$ is weakly homogeneous, $\PP'$ is rigid (having no non-trivial automorphisms), and both with the same Boolean completion (which is weakly homogeneous). 

\begin{theorem}\label{theorem:decisive}
Let $\BB$ be a complete Boolean algebra. The following are equivalent:
\begin{enumerate}
\item $\BB$ is weakly homogeneous.
\item If $G$ is $V$-generic for $\BB$ then for every $p\in\BB$ there is a $V$-generic $H$ such that $p\in H$ and $V[G]=V[H]$.
\item For every $\varphi(\dot x_1,\ldots,\dot x_n)$ in the language of forcing, and $u_1,\ldots,u_n\in V$, $1_\BB$ decides the truth value of $\varphi(\check u_1,\ldots,\check u_n)$.
\end{enumerate}
\end{theorem}
\begin{proof}The following is an outline of the proof. The implication from (1) to (3) is a well-known theorem. Assuming (3) holds, then taking $\varphi(\check p)$ to be ``There is a generic filter $H$ such that $p\in H$, and $\forall x$, $x\in V[H]$'', is decided by $1_\BB$ and of course it must be decided positively (genericity can be easily formulated using the set of dense subsets of $\BB$ in $V$ as a parameter of $\varphi$).

Finally, (2) to (1) follows from the theorem of Vop\v{e}nka and H\'ajek, that if $V[H]=V[G]$ where $G,H$ are $V$-generic filters for a complete Boolean algebra $\BB$, then there is an automorphism $\pi$ of $\BB$ such that $\pi''G=H$. In particular if $q\in G$ and $p\in H$ then $\pi q$ is compatible with $p$ (see \cite[Section~3.5,~Theorem~1]{Grigorieff1975} for a full discussion about the theorem).
\end{proof}

Of course, the second and third conditions may apply to any notion of forcing, and they are preserved by passing on to the Boolean completion. So checking them would suffice in order to show that the Boolean completion of $\PP$ is weakly homogeneous.

A natural weakening of the second condition is to say that for every $p\in\PP$ there is some $H$ which is $V$-generic and $p\in H$, without requiring $V[G]=V[H]$. It is unclear whether or not this weakening is indeed equivalent to the above conditions. But it does entail that the maximum condition of $\PP$ decides some statements about the ground model. More specifically if $\varphi(\dot x,\check u_1,\ldots,\check u_n)$ is a statement which is upwards absolute, then $1_\PP$ decides the truth value of $\forall x\varphi(x,\check u_1,\ldots,\check u_n)$. In particular ``$\kappa$ is singular'' is decided by $1_\PP$.

Further weakening the second condition is the following property:

\begin{definition}We say that a forcing $\PP$ is \textit{amalgamated} by a forcing $\QQ$ if for every $p_1,p_2\in\PP$ there exists some $q\in\QQ$ such that whenever $G$ is a $V$-generic filter for $\QQ$, and $q\in G$ then there are $H_1,H_2\in V[G]$ which are $V$-generic for $\PP$ and $p_i\in H_i$. If $\PP$ is amalgamated by itself then we say that it is \textit{amalgamable}.
\end{definition}

Some basic and immediate observations:
\begin{enumerate}
\item $\PP$ is amalgamated by $\PP\times\PP$.
\item If $\PP$ is amalgamated by $\QQ$ then for all $\RR$, $\PP$ is amalgamated by $\QQ\times\RR$.
\item If $\QQ$ amalgamates $\PP$ and $\QQ$ is $\lambda$-distributive, then $\PP$ is also $\lambda$-distributive.
\item Every weakly homogeneous forcing is amalgamable.
\item Radin forcing is amalgamable, although not weakly homogeneous.
\end{enumerate}

Let us examine a representative case for amalgamable forcings, the Magidor forcing which changes the cofinality of $\kappa$ to be $\omega_1$. This is a relatively ``nice'' forcing, which does not add new sets to very small cardinals and does not collapse any cardinals. It does, however, add new bounded subsets to $\kappa$. More specifically, it changes cofinality of many cardinals below $\kappa$ to be $\omega$. 

Can an amalgamable forcing change the cofinality of $\kappa$ to be $\omega_1$ without adding bounded subsets at all? The answer is negative, as shown in \autoref{Theorem:Reflective changes cofinalities}. To approach this problem, we weaken ``not adding bounded subsets''. Instead of not adding bounded sets, we want to be able and foresee possible changes to the stationary sets of the target cofinality. 

Returning to the Magidor forcing which changes the cofinality of $\kappa$ to be $\omega_1$, recall that we assume the existence of a sequence $\tup{U_\alpha\mid\alpha<\omega_1}$ of normal measures on $\kappa$, increasing in the Mitchell order. Given $X\subseteq\kappa$, we define $h(X)=\{\alpha<\omega_1\mid X\in U_\alpha\}$. This is a homomorphism of Boolean algebras. Moreover, if $\dot{c}$ is the canonical name of the generic club added by the forcing, $h(X) = \{\alpha < \omega_1 \mid\ \forces\dot{c}(\check\alpha) \in\check X\}$ up to a non-stationary subset of $\omega_1$. So $h$ predicts, in some sense, the structure of generic clubs in $\kappa$.

This gives us a definition in $V$ of an ideal on $\kappa$, $I=\{X\subseteq\kappa\mid h(X)\in\NS_{\omega_1}\}$. We can also define $I$ through the forcing relation and $\dot c$, $I=\{X\subseteq\kappa\mid\ \forces \dot c^{-1}(\check X)\in\check\NS_{\omega_1}\}$ (since the Magidor forcing does not add subsets to $\omega_1$, there is no confusion as to where $\NS_{\omega_1}$ is computed). If $h$ guessed what sort of changes the stationary sets undergo, $I$ predicts which sets will become non-stationary sets after forcing with the Magidor forcing. This property is generalized in the following definition.
 
\begin{definition}\label{Def:reflective forcing}Let $\PP$ be a forcing notion, $\kappa>\mu>\omega$ two regular cardinals. We say that $\PP$ is \textit{$(\kappa,\mu)$-reflective} if there exists a $\PP$-name $\dot c$ and an ideal $I\subseteq\power(\kappa)$ such that:
\begin{enumerate}
\item $\PP$ does not add new subsets to $\mu$,
\item $\forces_\PP\dot c\colon\check\mu\to\check\kappa \text{ is a continuous and cofinal function}$,
\item for every stationary $S\subseteq\kappa$ in the ground model, $\forces_\PP \check S\in\check I\leftrightarrow \dot c^{-1}(\check S) \in\check\NS_\mu$.
\end{enumerate}
It follows that $\power(\kappa)/ I$ has cardinality $<\kappa$ and therefore it is $\kappa$-c.c.\footnote{If $\kappa\leq 2^\mu$, then the cofinal sequence encodes a new subset of $\mu$ as a $\mu$-sequence in $2^{\mu\times\mu}$. Therefore $2^\mu<\kappa$. If $A\sdiff B\notin I$, then after forcing with $\mathbb{P}$, $c^{-1}(A) \neq c^{-1}(B)$ and the requirement on the chain condition is fulfilled.} We will call $\power (\kappa)/I$ the \textit{reflected forcing} of $\PP$. The remark above shows that $I$ is a very saturated ideal. As we will see shortly, the fact that $I$ is defined using the generic club $\dot{c}$ implies that $I$ is weakly normal (see \cite{Kanamori1976}) and $I$ is $\mu$-closed.    
\end{definition}
\begin{remark}
If $\PP$ is a $(\kappa,\mu)$-reflective forcing then $I$, as above, is unique. To see this, note that if $\dot c_1,\dot c_2$ are two names for continuous cofinal functions from $\mu$ to $\kappa$, then $\forces_\PP\rng\dot c_1\triangle\rng\dot c_2$ is non-stationary. Therefore the symmetric difference between $\dot c_1^{-1}(S)$ and $\dot c_2^{-1}(S)$ is non-stationary as well.
\end{remark}

We denote by $\RR_\PP$ the reflected forcing of $\PP$. Let us denote by $U$ the generic ultrafilter added by $\RR_\PP$. Finally, let us denote by $\tup{M_\PP,\E_\PP}$ the generic ultrapower of $V$ by $U$, $V^\kappa/U$ as defined in $V[U]$. This model need not be well-founded, since $U$ is often not $\sigma$-closed. If $F\colon S\to V$ is a function in $V$, where $S\subseteq\kappa$ is $I$-positive, we denote by $[F]_U$ the equivalence class of $F$ in the model $M$. When the context is clear, we will omit all the subscripts from the notations and write $\RR,M$, etc.

We aim to show that $(\kappa,\mu)$-reflective forcings must change cofinalities below $\kappa$, and that cofinality-changing forcing which are amalgamable and distributive are reflective. One immediate consequence would be that forcings which are slightly homogeneous and change cofinalities to an uncountable one, such as the Magidor or Radin forcings, must do ``some damage'' to the structure of cardinals below $\kappa$. This remark is also applicable to cofinality-changing forcings which collapse $\kappa$.

\begin{theorem}Suppose that $\kappa>\mu>\omega$ are regular cardinals, and let $\PP$ be a forcing changing the cofinality of $\kappa$ to be $\mu$. If $\PP$ can be amalgamated by a forcing $\QQ$ which does not add subsets to $\mu$, then $\PP$ is $(\kappa,\mu)$-reflective.
\end{theorem}
\begin{proof}
$\PP$ is amalgamated by $\QQ$ which does not add new subsets to $\mu$, therefore $\PP$ has this property as well. Let $\dot c$ be a name such that $\forces_\PP\dot c\colon\check\mu\to\check\kappa$ continuous and cofinal. Moreover $\forces_\QQ\cf(\check\kappa)=\check\mu$. This is because by adding any generic for $\PP$ we add a continuous cofinal function from $\mu$ to $\kappa$. If we had changed the cofinality of $\kappa$, we would have added a new cofinal sequence in $\mu$ (and in fact change its cofinality).

We define an ideal over $\kappa$:
\[I=\{X\subseteq\kappa\mid\ \forces_\PP\dot c^{-1}(\check X)\in\check\NS_\mu\}.\]
The first two conditions of \autoref{Def:reflective forcing} are automatically satisfied for $I$. It remains to verify the last condition. Let $S\subseteq\kappa$ be a stationary set in $V$. 

We claim that if there exists $p\in\PP$ such that $p\forces_\PP\dot c^{-1}(\check S)\in\check\NS_{\omega_1}$, then every condition must force that. Suppose that $p$ is such a condition, given any $p'\in\PP$ pick some $q\in\QQ$ and $H\subseteq\QQ$ which is $V$-generic with $q\in H$ such that there are $G,G'\in V[H]$ which are $V$-generic for $\PP$, and $p\in G$, $p'\in G'$.  Let $c=\dot c^G$ and $c'=\dot c^{G'}$ be the interpretation of the cofinal sequence under these generics. 

Since $\QQ$ did not add subsets to $\mu$, there is a club $E\in V$ such that $c\restriction E=c'\restriction E$ and $E\cap c^{-1}(S)=\varnothing$. Therefore ${c'}^{-1}(S)\cap E=c^{-1}(S)\cap E=\varnothing$, so $p'$ cannot force that $\dot c^{-1}(\check S)\text{ is stationary}$. And therefore no condition can force that, and indeed $\forces_\PP\dot c^{-1}(\check S)\in\check\NS_\mu$.  This completes the proof, since the other direction of the third condition always holds.
\end{proof}

\begin{corollary}If $\PP$ is a weakly homogeneous forcing which changes the cofinality of $\kappa$ to $\mu>\omega$ without adding new subsets of $\mu$, then $\PP$ is $(\kappa,\mu)$-reflective.\qed
\end{corollary}

From the following theorems we learn that a weakly homogeneous forcing which changes the cofinality of $\kappa$ to be uncountable, must add bounded sets. So the two nice properties do not co-exist with changing cofinality to be uncountable. In conjunction with the following theorem, it shows that indeed any forcing changing the cofinality of a regular cardinal $\kappa$ to be uncountable cannot be both homogeneous and preserve cofinalities below $\kappa$. In fact, even weakening ``not adding bounded subsets'' to ``$\kappa$ remains a strong limit cardinal'' implies that it is also not the least inaccessible (and in fact, not the least Mahlo cardinal either), as shown in \autoref{Theorem: omega_1 Mahlo}.

\begin{theorem}\label{Theorem:Reflective changes cofinalities}
If $\PP$ is $(\kappa,\mu)$-reflective, then $\PP$ must change the cofinality of some $\lambda<\kappa$.
\end{theorem}
\begin{proof}
Let $I$ be the ideal witnessing the reflective property, and $\dot c$ a name for a cofinal and continuous function from $\mu$ to $\kappa$ defining $I$. Forcing with the reflected forcing of $\RR$, we obtain a generic ultrafilter $U$ over $\kappa$. Let $M=M_\PP$, and $j\colon V\to M$ the ultrapower embedding. We digress from the main proof, to prove two lemmas about $M$.

\begin{lemma}\label{Lemma:diagonal}
Let $d$ be the diagonal function on $\kappa$, then $\sup j''\kappa=[d]$.
In other words, for every $M$-ordinal $[f]$, if $M\models[f] \E [d]$, then there is some $\beta < \kappa$ such that $M\models[f] \E j(\beta)$.
\end{lemma}
\begin{proof}[Proof of Lemma]
Suppose that $M\models [f]\E [d]$, then there is some $S\in U$ that for every $\alpha\in S$, $f(\alpha)<d(\alpha)=\alpha$. 
We will find $S_0\subseteq S$ such that $S_0\notin I$ and $f'' S_0 \subseteq \beta < \kappa$, and therefore $S_0\forces_\RR [\check f]\E j(\check\beta)$.

In $V[G]$, where $G$ is a generic filter for $\PP$, the function $f$ restricted to $S\cap\dot c^G$ is regressive on stationary set. Therefore there is a stationary $T\subseteq S\cap \dot c^G$ in $V[G]$ and $\beta<\kappa$, such that $f(\alpha)<\beta$ for every $\alpha\in T$. So in $V$ the set $S_0= f^{-1}(\beta)\notin I$.
\end{proof}

\begin{lemma}\label{Lemma:continuity}
Let $\alpha$ be an ordinal for which there is $p\in\PP$ such that $p\forces_\PP\cf(\check\alpha) < \check\mu$. Then $\sup j''\alpha = j(\alpha)$. In particular $M$ is well-founded at least up to $\mu$, and the critical point of $j$ is at least $\mu$.
\end{lemma}
\begin{proof}[Proof of Lemma]
Let $f\colon S\to\alpha$, for $S\in U$. We want to find $S_0 \subseteq S$, $S_0\notin I$ such that $f''S_0 \subseteq \beta$ for some $\beta<\alpha$. 

Let $\eta<\mu$, $p\forces_\PP\cf(\check\alpha)=\check\eta$ and $G$ be a $V$-generic filter for $\PP$ such that $p\in G$.
For every $\beta<\alpha$, let $S_\beta = \{\xi<\kappa\mid f(\xi)<\beta\}$. Working in $V[G]$, let $c=\dot c^G$ be the generic cofinal function from $\mu$ to $\kappa$, and let $\{\beta_i\mid i<\eta\}$ be a cofinal sequence in $\alpha$. We know that $c^{-1}(S)=\bigcup\{c^{-1}(S_{\beta_i})\mid i<\eta\}$, so it cannot be the case that for every $i$, $c^{-1}(S_{\beta_i})$ is non-stationary. Therefore at least one $\beta_i$ with $c^{-1}(S_{\beta_i})$ stationary. Therefore $S_{\beta_i}\notin I$ back in $V$. From this it follows by induction that $j(\alpha)=\alpha$ for all $\alpha<\mu$.
\end{proof}

We return to the proof of the theorem. Assume towards a contradiction that $\PP$ does not change cofinalities below $\kappa$. This implies that $S^\kappa_\omega\notin I$, since in that case if $G\subseteq\PP$ is $V$-generic, then $(S^\kappa_\omega)^V=(S^\kappa_\omega)^{V[G]}$ and therefore it is a stationary set, and in particular it meets $\rng \dot c^G$.

We may assume without loss of generality that $S^\kappa_\omega\in U$. We will show that \[M\models\cf([d])=\cf(\sup j''\kappa)=\omega,\] and therefore $V[U]\models\cf(\kappa)=\omega$ which is impossible since $\RR$ is $\kappa$-c.c., so $\forces_\RR\cf(\check\kappa)=\check\kappa$ (because a $\kappa$-c.c.\ forcing cannot change the cofinality of a regular cardinal $\kappa$).

Now pick some $F\colon S^\kappa_\omega\to\kappa^\omega$ such that $F(\alpha)$ is a sequence cofinal in $\alpha$. Since $S^\kappa_\omega \in U$, by \Los theorem $[F]$ is a cofinal $j(\omega)$ sequence at $[d]$ in $M$. By \autoref{Lemma:continuity}, $j(\omega)=\omega$, and therefore $M\models [F] = \left\{[F](n)\midd n < \omega \right\}$ and $M\models \sup\{[F](n)\mid n<\omega\} = [d]$. Next, working in $V[U]$, from the first lemma it follows that for each $n$ there is some $\alpha_n<\kappa$ such that $M\models [F](n)\leq j(\alpha_n)$, and therefore $\tup{\alpha_n\midd n<\omega}$ is a sequence cofinal in $\kappa$. This is a contradiction, as mentioned above, and therefore we changed some cofinalities below $\kappa$.
\end{proof}

The following is a relatively straightforward consequence of the above proof. We extend it in \autoref{Theorem: omega_1 Mahlo}, but it is worth proving on its own.

\begin{corollary}\label{Corollary:Mahlo}
Suppose $\PP$ is a $(\kappa,\mu)$-reflective forcing such that $\forces_\PP\check\kappa$ is a strong limit cardinal. Then $\kappa$ is a Mahlo cardinal in $V$.
\end{corollary}
\begin{proof}
We aim to show that $S=\kappa\cap\Reg$ is stationary in $\kappa$, and moreover $\kappa\setminus S\in I$. This is equivalent to saying that for every $U$ induced by the generic filter for the reflected forcing, $\RR$, we have that $V^\kappa/U=M\models\cf([d])=[d]$. This is a consequence of \Los theorem, similar to the previous proof. So it is enough to show that for each $\alpha<\kappa$, $M\models j(\alpha)<\cf([d])$.

From \autoref{Lemma:magidor} it will follow that $V[U]\models |j(\alpha)|<\kappa$. By the $\kappa$-c.c.\ of $\RR$ we know that in $V[U]$ the cofinality of $\{x\mid M\models x<[d]\}$ (as a linear order, which is not necessarily well-founded) is $\kappa$. 
\end{proof}
\begin{lemma}[Magidor]\label{Lemma:magidor}
Assume that $\PP$ is a $(\kappa,\mu)$-reflective forcing and $\RR$ its reflected forcing. Let $U\subseteq\RR$ and $G\subseteq\PP$ be $V$-generic filters. Then $V[U]\models|j(\alpha)|\leq(\alpha^\mu)^{V[G]}$.
\end{lemma}
\begin{proof}
We define an equivalence relation on functions from $\kappa$ to $\alpha$:\[F\sim G\iff \{\alpha<\kappa\mid F(\alpha)\neq G(\alpha)\})\in I.\]
We have that $F\sim G$ if and only if there exists $p\in\PP$ such that $p\forces_\PP\check F\res\rng\dot c=_{\NS_\kappa}\check G\res\rng\dot c$, equivalently $F\nsim G$ if and only if $1_\PP\forces\check F\res\dot c\neq_{\NS_\kappa}\check G\res\dot c$. And we have that there are only $(\alpha^\mu)^{V[G]}$ possible values for this restriction.
\end{proof}

Recall the definition for a $\square(\kappa)$ sequence, where $\kappa$ an infinite cardinal. A sequence $\tup{C_\alpha\midd\alpha<\kappa}$ is a $\square(\kappa)$ sequence if:
\begin{enumerate}
\item For each $\alpha$, $C_\alpha$ is a club in $\alpha$.
\item For each $\alpha<\beta$, if $\alpha\in\acc C_\beta$, then $C_\beta\cap\alpha = C_\alpha$.
\item For each $\alpha$, $C_{\alpha+1}=\{\alpha\}$.
\item There is no club $A\subseteq\kappa$ such that for $\alpha\in\acc A$, $A\cap\alpha=C_\alpha$ (such $A$ is called \textit{a thread}).
\end{enumerate} 
We say that $\square(\kappa)$ holds, if there is a $\square(\kappa)$ sequence.\footnote{For a historical overview of $\square(\kappa)$ and some related results see \cite{RinotSquare}.} 

\begin{theorem}\label{Theorem: omega_1 Mahlo}
Let $\PP$ be a $(\kappa,\mu)$-reflective forcing, and assume that $\forces_\PP\check\kappa$ is strong limit.
We define in $V$ the following stationary $T = \{\alpha < \kappa \mid \exists p\in \PP : p\forces_\PP \cf (\check\alpha) < \check\mu\}$. In $V$ the following statements hold:
\begin{itemize}
\item For every $\eta<\mu$ and $\tup{S_\alpha\midd\alpha<\eta}$ stationary subsets of $T$ have a common reflection point.
\item $\kappa$ is $\mu$-Mahlo.
\item $\square(\kappa)$ fails. 
\end{itemize}
\end{theorem}
\begin{remark}
In general, the first assertion in \autoref{Theorem: omega_1 Mahlo} implies the third one, since $\square(\kappa)$ implies the existence of pairs of stationary set without common reflection point (see \cite[Lemma~3.2]{RinotSquare}). 
\end{remark}
\begin{proof}
Let $U$ be the generic ultrafilter added by the reflected forcing of $\PP$, and let $M=V^\kappa/U$ and $j$ the ultrapower embedding. Let $\eta < \mu$, $\tup{ S_i \mid i < \eta}$ be a sequence of stationary subsets of $T$. Since $\eta < \mu$, by \autoref{Lemma:continuity}, $j(\eta) = \eta$. 
Let us show that for every $i$, $M \models j(S_i)\cap [d]$ is stationary, where $d(\alpha)=\alpha$ for $\alpha<\kappa$. Assume otherwise, then there is some $i< \eta$ and some  $C\in M$ a club in $M$ such that $M\models C\cap j(S_i)\cap [d] = \varnothing$. 
Let $E = \{\alpha \mid j(\alpha)\in C\}$, we claim that $E$ is a $T$-club in $V[U]$. We start by showing that $E$ is unbounded. 
Let $\alpha_0 < \kappa$, and pick by induction $\xi_n \in C$ and $\alpha_n < \kappa$ such that $\xi_n < j(\alpha_{n+1}) < \xi_{n+1}$. Note that by the $\kappa$-c.c.\ of $\RR$ we may  pick those $\alpha_n$ in $V$. So for $\alpha_\omega = \sup\{\alpha_n\mid n<\omega\}$, $j(\alpha_\omega) \in C$ since $\forall \zeta < j(\alpha_\omega) \exists \xi > \zeta, \xi \in C$, since $j(\alpha_\omega) = \sup\{j(\alpha_n)\mid n\in\omega\}$. 

Now let us show that $E$ is $T$-closed, i.e.\ for every $\alpha \in T$ if $E\cap \alpha$ is unbounded ,then $\alpha \in E$. Again, by \autoref{Lemma:continuity}, if $\alpha \in T$ then $\sup j''\alpha = j(\alpha)$ so if $E \cap \alpha$ is unbounded then $M\models \forall \xi < j(\alpha) \exists \zeta\in C\setminus\xi$ and therefore $j(\alpha) \in C$, as wanted.

But $E \cap S_i = \varnothing$ by the definition of $E$, and this is impossible since $S_i$ in stationary in $V[U]$, because $\RR$ is $\kappa$-c.c. Therefore in $M$ there is some $M$-ordinal below $j(\kappa)$ which reflects all the stationary sets in $j(\tup{S_\alpha\mid\alpha<\eta})$, and by elementarity the sequence has a common reflection point below $\kappa$, as wanted.

Note that the common reflection point is in fact $[d]$. By the proof of \autoref{Corollary:Mahlo}, the set $S = \kappa \cap \Reg \in U$, so by \Los theorem there is a regular common reflection point. Moreover, $\kappa\setminus T\in I$ since every element in the club $\acc(\dot c)$ is of cofinality less than $\mu$. So we have that $S\cap T \in U$ which again by \Los theorem implies that every ${<}\mu$ stationary subsets of $T\cap \Reg$ reflect in some $\nu\in T\cap \Reg$. Now, by induction on $\alpha < \mu$ we get that $\kappa$ is $\mu$-Mahlo.

To see that $\square(\kappa)$ fails, assume towards contradiction that $\cC = \tup{C_\alpha \mid \alpha < \kappa}$ is a $\square(\kappa)$ sequence in $V$. Therefore $M \models ``j(\cC)$ is a $\square(j(\kappa))$-sequence'', so it has a $[d]$-th element. Let $C = j(\cC)([d])$ and as above, define $E$ to be $\{\alpha < \kappa \mid j(\alpha)\in C\}$, so $E$ is a $T$-closed set. For every $\alpha$, if $\cf^V(\alpha) = \omega$ then $\alpha \in T$. Since $\RR$ is $\kappa$-c.c., $(S^\kappa_\omega)^V$ is stationary. Therefore there are unboundedly many $\alpha$ such that $\alpha \in \acc E$ and $\cf^V(\alpha) = \omega$, so $j(\alpha) = \sup j''\alpha$. For every such $\alpha$, $C\cap j(\alpha)=D_\alpha = j(\cC)({j(\alpha)}) = j(C_\alpha)$. For every $\alpha < \beta$ in $\acc E \cap (S^\kappa_\omega)^V$, $C_\alpha$ is an initial segment of $C_\beta$ and therefore $F = \bigcup\{C_\alpha\mid\alpha \in \acc E \cap (S^\kappa_\omega)^V\}$ is a thread added by a $\kappa$-c.c.\ forcing. This is impossible as the next lemma shows.
\end{proof}

\begin{lemma}
Suppose that $\cC=\tup{C_\alpha\mid\alpha<\kappa}$ is a $\square(\kappa)$ sequence. If $\PP$ is $\kappa$-c.c., then $\PP$ does not add a thread to $\cC$.
\end{lemma}
\begin{proof}
Assume otherwise, and let $T$ be the thread added by $\PP$. Namely, $T$ is a club such that for every $\alpha\in\acc T$, $T\cap\alpha\in\cC_\alpha$. Since $\PP$ is $\kappa$-c.c.\ there is some club $T'\subseteq T$ in the ground model. 

For each $\alpha\in\acc T'$, $\forces_\PP\check\alpha\in\acc \dot{T}$ and therefore $\forces_\PP\check{C_\alpha} =\dot T\cap \check\alpha$. In particular, if $\alpha < \beta \in \acc T'$ then $C_\alpha \trianglelefteq C_\beta$ and therefore $E = \bigcup_{\alpha\in \acc T'} C_\alpha$ is a thread for $\cC$ in $V$. 
\end{proof}

\section{Some Questions}

\begin{question}In \autoref{theorem:pcf} we saw that we cannot change the cofinalities of infinitely many cardinals to a constant cofinality without adding $\omega$-sequences, or changing the cofinalities of all the cardinals in the $\pcf$. Is it possible to do just without adding $\omega$-sequences? As remarked after the proof, this seems to require much larger cardinals than just measurables of high Mitchell order. 
\end{question}

\begin{question}Continuing from the previous question, we might ask what happens if we have $\PP$ which changes the cofinality of a regular $\kappa_n$ to $\mu_n$, rather than a constant $\mu$. Will this forcing necessarily add an $\omega$-sequence, or change the cofinalities of the entire $\pcf$?
\end{question}

We have seen that there is no $(\kappa,\mu)$-reflective forcing where $\kappa$ is the least inaccessible, and in particular there is no weakly homogeneous forcing that changes the cofinality of the first inaccessible to $\omega_1$ without adding subsets to $\omega_1$. The following conjectures set the conditions imposed in the main theorems of \autoref{Section:Reflective} are somewhat optimal.

\begin{conjecture} It is consistent, relative to the existence of large cardinals, that there is a weakly homogeneous forcing which changes the cofinality of the first inaccessible to $\omega$ without adding bounded subsets.
\end{conjecture}

\begin{conjecture} It is consistent, relative to the existence of large cardinals, that there is $(\kappa,\omega_1)$-reflective forcing where $\kappa$ is the first $\omega_1$-Mahlo cardinal.   
\end{conjecture}

\section{Acknowledgments}
The authors would like to thank the following people: Menachem Magidor for many helpful discussions, and in particular for contributing the proof for \autoref{Lemma:magidor} which is a crucial ingredient in the proof of \autoref{Corollary:Mahlo}; Joel D.\ Hamkins for his helpful observation regarding the proof of \autoref{theorem:decisive}, that deciding statements about the ground model is equivalent to having many generic filters, and therefore the Vop\v{e}nka-Haj\'ek theorem can be applied; Monro Eskew for asking the question leading to \autoref{Corollary:Singulars} which provided a different context for applying the first theorem, and for his help in shaping the formulation of said theorem; and finally to Assaf Rinot for his suggestion that the statement of \autoref{Theorem:proper} be changed to match the content of the proof.

\providecommand{\bysame}{\leavevmode\hbox to3em{\hrulefill}\thinspace}
\providecommand{\MR}{\relax\ifhmode\unskip\space\fi MR }
\providecommand{\MRhref}[2]{%
  \href{http://www.ams.org/mathscinet-getitem?mr=#1}{#2}
}
\providecommand{\href}[2]{#2}

\end{document}